\documentclass[a4paper,12pt,reqno]{amsart}

\usepackage{amsmath,amssymb,amsfonts,latexsym}

\newtheorem{theorem}{Theorem}[section]

\newtheorem{lemma}{Lemma}[section]
\numberwithin{equation}{section}

\usepackage{comment}
\usepackage{marvosym}
\usepackage[margin=1.5 in]{geometry} 
\usepackage[colorlinks=true,
linkcolor=blue,     
citecolor=red,      
urlcolor=magenta,   
filecolor=cyan]{hyperref} 

\usepackage{color}
\usepackage{xcolor}
\usepackage[normalem]{ulem} 
\usepackage{soul}

\begin{document}
	\begin{center}{\bf \large
			{ADDITIONAL CONGRUENCES FOR GENERALIZED COLOR PARTITIONS OF HIRSCHHORN AND SELLERS \\[0.15cm] }  }\vspace{0.2cm}
	\end{center}

	\begin{center}
		\footnotemark[1]
		\bf Anjelin Mariya Johnson, 
	\footnotemark[2]
	\bf James A. Sellers, 
		\footnotemark[3]
		\bf S. N. Fathima  

	\end{center}
	\vspace{1 cm}
	\begin{center}
		\begin{minipage}{0.85\textwidth}  
			
	\noindent{\bf Abstract:} 
	Let $a_k(n)$ denote the number of partitions of $n$ wherein even parts come in only one color, while the odd parts may be ``colored" with one of $k$ colors, for fixed $k$. In this note, we find some congruences for $a_k(n)$ in the spirit of Ramanujan's congruences. We prove a number of results for $a_k(n)$ modulo powers of $2$ for infinitely many values of $k$. Our approach is truly elementary, relying on generating function manipulations, theta functions and $q$-dissection techniques. We then close by demonstrating an infinite family of congruences modulo 11 which is proven using a result of Ahlgren.
 \\
	\vspace{.125cm}
	
	\noindent {\bf \small Keywords} : Congruences, Partitions, Generating Function.\\
	\vspace{.05cm}
	
	\noindent {\bf \small Mathematical Subject Classification (2020)} : 05A17, 05A15, 11P83.
		\end{minipage}
\end{center}

\bigskip

	\section {Introduction}

A partition of a non-negative integer $n$ is represented by a sequence of positive integers arranged in non-increasing order whose sum equals $n$. The number of such partitions is denoted by $p(n)$, and its generating function is given by
	 \begin{align*}
	 	\sum_{n=0}^{\infty}p(n)q^n=\frac{1}{f_1},
	\end{align*}
	 where for $|q|<1$, we adopt the notations $(a;q)_\infty=\prod_{i\geq0}^{}(1-aq^i)$, and $f_\ell:=(q^\ell;q^\ell)_\infty$.
	Approximately 100 years ago,  Ramanujan \cite{r12, r13, r14} discovered amazing congruence properties satisfied by $p(n)$: 
	For all $n \geq 0$,
\begin{align}
	p(5n+4) \equiv& \;0 \pmod{5},\label{R1}\\
	p(7n+5) \equiv&\; 0 \pmod{7},\label{R2}\\
	p(11n+6) \equiv&\; 0 \pmod{11} \label{R3}.
\end{align}
	 Inspired by these results, the subject captured the interest of researchers leading to numerous contributions for various analogous partition functions from an arithmetic perspective. For a general overview of the theory of partitions, we refer the reader to the book of Andrews \cite{r2}.

Very recently, an overreaching area of study in partitions is  that of ``colored partitions" (partitions of positive integer into parts with colors). For instance, let $a(n)$ enumerate the partitions of $n$ wherein even part is monochromatic while the odd parts may appear in one of three colors.  This partition function $a(n)$  has several combinatorial interpretations (see \cite{r1}). For instance, $a(3)=16$, where the partitions in question can be written with colors denoted by subscripts:  
	\begin{align*}
	& 3_1,\; 3_2,\; 3_3,\; 2+1_1,\; 2+1_2,\; 2+1_3, \;1_1+1_1+1_1,\;1_1+1_2+1_2,\;\\ 
      & 1_1+1_2+1_3,\; 1_1+1_3+1_3, \;1_2+1_2+1_2,\; 1_2+1_2+1_3,\;\\ 
     &1_2+1_3+1_3,\; 1_3+1_3+1_3,\; 1_1+1_1+1_2,\; 1_1+1_1+1_3
\end{align*}
	The generating function for $a(n)$ is given by 
	\begin{align*}
		\sum_{n=0}^{\infty} {a(n)q^n}=\frac{f_2^2}{f_1^3}.
	\end{align*}
	Amdeberhan and Merca \cite{r1} succeeded in proving that, for all $n\geq 0$, 
\begin{align*}
		a(7n+2) \equiv 0 \pmod{7}.
\end{align*}
	The partition function $a(n)$ is listed in the Online Encyclopedia of Integer Sequences (OEIS) \cite[A298311]{r15}.  
	With the goal to generalize the idea further, Hirschhorn and Sellers \cite{r4} defined an infinite family of functions $a_k(n)$  that counts the number of partitions of $n$ wherein even parts are monochromatic while the odd parts may appear in one of $k$ colors for fixed $k\geq 1$. The generating function for $a_k(n)$ is given by
	\begin{align}
	\sum_{n=0}^{\infty}a_k(n)q^n=\frac{f_2^{k-1}}{f_1^k}.\label{1.a}
\end{align}\\
	 In \cite{r4}, Hirschhorn and Sellers utilized elementary $q$-series techniques to prove the following infinite family of congruences modulo 7.\\
	 \begin{theorem}[\cite{r4}, Corollary 3.1]\label{t:2.1.1} 
	 $For \; all \; j\geq0 \; and \; all \; n\geq0$,\\
	 \begin{align}
	 a_{7j+1}(7n+5)\equiv 0& \pmod{7}\label{c1}\\
	 a_{7j+3}(7n+2)\equiv 0& \pmod{7}\label{c2}\\
	 a_{7j+4}(7n+4)\equiv 0& \pmod{7}\label{c3}\\
	 a_{7j+5}(7n+6)\equiv 0& \pmod{7}\label{c4}\\
	 a_{7j+7}(7n+3)\equiv 0& \pmod{7}\label{c5}.
	  \end{align}
	  	\end{theorem}
\noindent Further, Thejitha and Fathima \cite{r10} employed the theory of modular forms to obtain the following two sets of congruences satisfied by the function $a_5(n)$.
	\begin{theorem}\label{t:2.1.2}
		For all $n\geq0,$
		\begin{align*}
			a_5(5n+3)\equiv0 \pmod{5}.
		\end{align*}
	\end{theorem}
	\begin{theorem}\label{t:2.1.3}
		For all $\alpha \geq0$ and all $n\geq0,$
		\begin{align*}
			a_5\left(3^{2\alpha+2}n+\frac{153\cdot 3^{2\alpha}-1}{8}\right)\equiv 0\pmod{3}.
		\end{align*}
	\end{theorem}
	In a manner similar to that which was highlighted in \cite{r4}, Sellers \cite{r8} obtained the following theorem.  
	\begin{theorem}[\cite{r8}, Corollary 3.1]\label{t:2.1.4}
	For all $j \geq0$ and all $n\geq0,$
	\begin{align*}
		a_{5j+5}(5n+3)\equiv0\pmod{5}.
	\end{align*}
	\end{theorem}

	Using theta functions and elementary generating function manipulations, our overarching goal in this article is to further study $a_k(n)$ from an arithmetic perspective. With this in mind, we prove several new congruences modulo powers of $2$. The following are our main results.
	\begin{theorem}\label{t:2.1.5}
		For all $r \geq 1$ and all $n \geq 0$, 
		\begin{align*}
			a_{2^r}(2n+1) \equiv 0 \pmod {2^r}.
		\end{align*}
	\end{theorem}
	
	\begin{theorem}\label{t:2.1.6}
		For all $r \geq 1$ and all $n \geq 0$,  
		\begin{align*}
			a_{2r}(2n+1) \equiv 0 \pmod {2}.
		\end{align*}
	\end{theorem}


     \begin{theorem}\label{t:2.1.7}
	Let $p$ be prime, $p > 3$, and $r$, $0 < r < p$, such that $24r+1$ is a quadratic nonresidue
	modulo $p$. Then, for all $n \geq 0$, 
	\begin{align*}
		a_3(pn+r) \equiv 0 \pmod{2}.
	\end{align*}
	\end{theorem}

	\begin{theorem}\label{t:2.1.8}
	Let $p$ be prime, $p > 2$, and $r$, $0 < r < p$, such that $8r+1$ is a quadratic nonresidue
			modulo $p$. Then, for all $n \geq 0$, 
			\begin{align*}
				a_5(pn+r) \equiv 0 \pmod{2}.
			\end{align*}
					\end{theorem}
\begin{theorem}\label{t:2.1.10}
			For all $n \geq 0$,
		\begin{align*}
		a_{11}(9n+6) &\equiv 0 \pmod {2}.
		\end{align*}
\end{theorem}

	 Section 2 is devoted to providing the tools necessary for the proofs of our main results, and these proofs are then provided in Section 3. In view of Theorem \ref{t:2.1.1} we conclude this paper  by proving a set of congruences modulo 11.

\section{Preliminaries}
We require the following results and lemmas related to Ramanujan's theta functions. 
Ramanujan's general theta function is 
\begin{align*}
	f(a,b):= \sum_{n=-\infty}^{\infty}a^{\frac{n(n+1)}{2}}b^{\frac{n(n-1)}{2}},\;\;|ab|<1.
\end{align*}
In Ramanujan's notation,  Jacobi's triple product identity (\cite{r3}, Entry 19) is given by
\begin{align*}
	f(a,b)=(-a;ab)_\infty(-b;ab)_\infty(ab;ab)_\infty.
\end{align*}
In what follows we give special cases of Jacobi's triple product identity.

\bigskip 
\begin{lemma}[Euler's Pentagonal Number Theorem (\cite{r3}, Entry 22)] 
\label{lem:PNT}
We have
	\begin{align*}
		f_1=\sum_{k=-\infty}^{\infty}(-1)^kq^\frac{k(3k-1)}{2}.
	\end{align*}
\end{lemma}
\begin{lemma}[Jacobi's Theorem (\cite{r5}, (1.7.1))]
\label{2.8}
We have
\begin{align*}
	f_1^{3} 
	\notag=& \sum_{k=0}^{\infty} (-1)^{k} (2k + 1) q^{\frac{k(k+1)}{2}}.
	\end{align*}
\end{lemma}
We next highlight the following 3-dissections which allow us to write the necessary generating functions in a suitable manner in order to prove our main results.

\begin{lemma}[Hirschhorn \cite{r5}, Section 14.3]\label{l3}
	We have
	\begin{align*}
		\frac{f_{2}^{2}}{f_{1}}	= \frac{f_{6} f_{9}^{2}}{f_{3} f_{18}}+q\frac{f_{18}^{2}}{f_{9}}.
	\end{align*}
	\end{lemma}
		\begin{lemma}[Hirschhorn \cite{r5}, Section 14.8]\label{l4}
			We have
			\begin{align*}
		f_1^3 =\frac{f_6f_9^6}{f_3f_{18}^3}+qf_9^3.
		\end{align*}
\end{lemma}

\noindent 
We close this section with a pivotal congruence result which can be established using the binomial theorem.
\begin{lemma}\label{bt}
For any prime $p$ and positive integers $k$ and $m$,
	\begin{align*}
	f_m^{p^k} \equiv f_{mp}^{p^{k-1}} \pmod{p^k}.
\end{align*}
\end{lemma}

	\section{Proof of Theorems }
	
\noindent 
We are now prepared to prove our main results.
	\begin{proof}[\it\textbf{Proof of Theorem \ref{t:2.1.5}}]
	Thanks to (\ref{1.a}), we can rewrite the generating function for $a_k(n)$ as 
	\begin{align*}
		\sum_{n=0}^{\infty}a_{2^r}(n)q^n=&\frac{f_2^{{2^r}-1}}{f_1^{2^r}}.
	\end{align*}
	Lemma \ref{bt} allows us to see immediately that
	\begin{align}
	\sum_{n=0}^{\infty}a_{2^r}(n)q^n\equiv& f_2^{{2^{r-1}}-1}\pmod{2^r}.\label{3.a}
	\end{align}
Note that $f_2^{{2^{r-1}}-1}$ is a function of $q^2$, so our theorem follows immediately.
	\end{proof}
	\begin{proof}[\it\textbf{Proof of Theorem \ref{t:2.1.6}}]Using (\ref{1.a}) and Lemma \ref{bt} we have
	\begin{align*}
		\sum_{n=0}^{\infty}a_{2r}(n)q^n
		\equiv& \frac{f_2^{2r-1}}{f_2^{r}}\pmod{2} \\
           =& f_2^{r-1}.  
	\end{align*}
Note that $f_2^{r-1}$ is a function of $q^2$, so our theorem follows immediately.
	\end{proof}
	\begin{proof}[\it\textbf{Proof of Theorem \ref{t:2.1.7}}]Using (\ref{1.a})
		 we have
		\begin{align*}
			\sum_{n=0}^{\infty}a_3(n)q^n=&\frac{f_2^2}{f_1^3}\\
			\equiv& \frac{f_1^4}{f_1^3}\pmod{2}\\
			=&f_1\\
			\equiv&\sum_{k =-\infty}^\infty q^{\frac{k(3k+1)}{2}} \pmod{2}
		\end{align*}
		thanks to Lemma \ref{lem:PNT}.
		Now, we are interested in comparing the exponents of $q^{pn+r}$ and $q^{k(3k+1)/2}$. That
		is, we want to know if it is ever possible to have
		\begin{align*}
			pn+r = \frac{k(3k+1)}{2}
		\end{align*}
		for specific $n$ and $k$. Note that, if this is possible, then we must have
		\begin{align*}
			r \equiv \frac{k(3k+1)}{2} \pmod{p}.
		\end{align*}
		Completing the square, this means we would have
		\begin{align*}
			24r + 1 \equiv (6k + 1)^2 \pmod{p}.
		\end{align*}
		However, this is never possible because we have assumed that $24r+1$ is a quadratic
		nonresidue modulo $p$. This completes the proof of Theorem \ref{t:2.1.7}.			
	\end{proof}
	\begin{proof}[\it\textbf{Proof of Theorem \ref{t:2.1.8}}]
We have
\begin{align*}
	\sum_{n=0}^{\infty} a_5(n) q^n 
	=& \frac{f_2^4}{f_1^5} \\
	\equiv& \frac{f_1^8}{f_1^5}\pmod{2} \\
	=& f_1^3 \\
	\equiv& \sum_{k=0}^{\infty} q^{\frac{k(k+1)}{2}} \pmod{2}
\end{align*}
thanks to Lemma \ref{2.8}.
Now, we are interested in comparing the exponents of $q^{pn+r}$ and $q^{\frac{k(k+1)}{2}}$. That
is, we want to know if it is ever possible to have
\begin{align*}
	pn+r = \frac{k(k+1)}{2}
\end{align*}
for specific $n$ and $k$. Note that, if this is possible, then we must have
\begin{align*}
	r \equiv \frac{k(k+1)}{2} \pmod{p}.
\end{align*}
Completing the square, this means we would have
\begin{align*}
	8r + 1 \equiv (2k + 1)^2 \pmod{p}.
\end{align*}
However, this is never possible because we have assumed that $8r+1$ is a quadr]atic
nonresidue modulo $p$. This completes the proof.
\end{proof}

\begin{proof}[\it\textbf{Proof of Theorem \ref{t:2.1.10}}]
We note that 
\begin{align*}
\sum_{n=0}^{\infty}a_{11}(n)q^n
&=
\frac{f_2^{10}}{f_1^{11}} \\
&=
\frac{f_2^{10}}{f_1^{14}}\cdot f_1^3 \\
&\equiv 
\frac{f_2^{10}}{f_2^{7}}\cdot f_1^3 \pmod{2}\\
&= 
f_1^3f_2^3  \\
&= 
\left( \frac{f_6f_9^6}{f_3f_{18}^3}+qf_9^3 \right)\left( \frac{f_{12}f_{18}^6}{f_6f_{36}^3}+q^2f_{18}^3 \right) 
\end{align*}
thanks to Lemma \ref{l4}.  
Thus, 
$$
\sum_{n=0}^{\infty}a_{11}(3n)q^{3n} \equiv \frac{f_6f_9^6}{f_3f_{18}^3}\cdot \frac{f_{12}f_{18}^6}{f_6f_{36}^3} +q^3f_9^3f_{18}^3  \pmod{2} 
$$
or 
\begin{align*}
\sum_{n=0}^{\infty}a_{11}(3n)q^{n} 
&\equiv 
\frac{f_3^6f_{4}f_{6}^3}{f_1f_{12}^3} +qf_3^3f_{6}^3  \pmod{2}  \\
&\equiv
\frac{f_3^6f_{2}^2f_{6}^3}{f_1f_{12}^3} +qf_3^3f_{6}^3 \pmod{2} \\
&= 
\frac{f_3^6f_{6}^3}{f_{12}^3}\cdot \frac{f_{2}^2}{f_1} +qf_3^3f_{6}^3 \\
&= 
\frac{f_3^6f_{6}^3}{f_{12}^3} \left( \frac{f_{6} f_{9}^{2}}{f_{3} f_{18}}+q\frac{f_{18}^{2}}{f_9} \right) +qf_3^3f_{6}^3
\end{align*}
thanks to Lemma \ref{l3}.  
Note that no terms of the form $q^{3n+2}$ will arise from the power series representation of the final expression above.  Therefore, for all $n\geq 0$, 
$$
a_{11}(3(3n+2)) = a_{11}(9n+6) \equiv 0 \pmod{2}.
$$
\end{proof}

\section{Congruences Analogous to Theorem \ref{t:2.1.1} }
In this section we prove a set of congruences modulo 11 which are similar to the congruences found in Theorem \ref{t:2.1.1}.
\bigskip

\noindent It is worthwhile to share the following results of Ahlgren \cite{r11} and Cooper, Hirschhorn, and Lewis \cite{rc1}.
\begin{theorem}[\cite{r11},\cite{rc1}]\label{imp}
	Suppose that $r$ and $s$ are integers and that $p$ is prime.  
	Let $f_1^r f_2^s =\sum_{n=0}^{\infty} c(n) q^n$. Then the coefficients $c(n)$ satisfy
	\begin{align*}
		c\left( pn + \frac{(r + 2s)(p^2 - 1)}{24}\right)=\epsilon
		p^{\frac{r+s}{2}-1} c\left(\frac{n}{p}\right),
	\end{align*}
	for the following values of $r$, $s$, $p$, and $\epsilon$.
	\begin{table}[h]
		\centering
		\begin{tabular}{|c|c|c|}
			\hline
			$(r,s)$ & $p$ & $\epsilon$  \\ 
			\hline
			$(7, 3), (3, 7)$ & $p\equiv7, 11\pmod{12} $& 
			$\begin{cases}
				1, & \text{if } p\equiv7\pmod{8}, \\
				-1, & \text{if } p\equiv3\pmod{8}
			\end{cases}$\\
			$(-6, 16)$ & $p\equiv 11\pmod{12}$ & 1\\ 
			$(11, -5)$ & $p\equiv7, 11\pmod{12} $& 
			$\begin{cases}
				1, & \text{if } p\equiv7\pmod{8}, \\
				-1, & \text{if } p\equiv3\pmod{8}
			\end{cases}$\\
			$(3,3)$ & $p\equiv 5,7\pmod{8}$ & 1\\ 
			$(2, 4)$ & $p\equiv7, 11\pmod{12} $&1\\
			\hline
		\end{tabular}
	\end{table}
	
\end{theorem}
\noindent With the help of the above theorem we prove congruences for $a_{11}(n)$  modulo 11.
\begin{theorem}\label{t:2.1.9}
	For all $n \geq 0$,  
	\begin{align}
		a_{11j+1}(11n+6)\equiv 0& \pmod{11}\label{1.5},\\
		a_{11j+4}(11n+10)\equiv 0& \pmod{11}\label{1.6},\\
		a_{11j+6}(11n+9)\equiv 0& \pmod{11}\label{1.7},\\
		a_{11j+8}(11n+8)\equiv 0& \pmod{11}\label{1.8},\;\; and\\
		a_{11j+11}(11n+1)\equiv 0& \pmod{11}\label{1.9}.
	\end{align}
\end{theorem}
\begin{proof}[\it\textbf{Proof}]When $j=0$, Ramanujan's congruence (\ref{R3}) implies congruence (\ref{1.5}).  Next, we consider (\ref{1.9}). 
	
	Thanks to (\ref{1.a}), we have
	\begin{align*}
		\sum_{n=0}^{\infty}a_{11}(n)q^n=&\frac{f_2^{10}}{f_1^{11}}\\
		\equiv &\frac{f_{22}}{f_{11}}\sum_{n=0}^{\infty}p(n)q^{2n}\pmod{11}.
	\end{align*}
	To conclude the proof of (\ref{1.9}) we note that $2n\equiv1\pmod{11}$ implies $n\equiv6\pmod{11}$. Again thanks to (\ref{R3}), we complete the proof of congruence (\ref{1.9}) of Theorem \ref{t:2.1.9}, when $j=0$.\\
	To prove congruence (\ref{1.6}), we consider
	\begin{align*}
		\sum_{n=0}^{\infty}a_{4}(n)q^n=&\frac{f_2^{3}}{f_1^{4}}\\
		=&\frac{f_2^{3}}{f_1^{4}}.\frac{f_1^7}{f_1^7}\\
		=&\frac{f_1^7f_2^3}{f_1^{11}}\\
		\equiv&\frac{f_1^7f_2^3}{f_{11}}\pmod{11}.
	\end{align*}
	Define
	\begin{align*}
		\sum_{n=0}^{\infty}A(n)q^n := {f_1^7}{f_2^3}.
	\end{align*}
	Since $p=11$ satisfies the condition $p\equiv7, 11\pmod{12}$ and further $p\equiv3\pmod{8}$, we set $r=7$ and $s=3$ in Theorem \ref{imp} to obtain 
	$$A\left( 11n + \frac{13\cdot 120}{24}\right)=-{11}^4 A\left(\frac{n}{11}\right).$$
	This implies
	\begin{align*}
		A\left( 11n + 65\right)\equiv&\; 0\pmod{11}
	\end{align*}
	for all $n\geq 0$.  
	It is easy to verify  using Mathematica that $A(m)\equiv 0 \pmod{11}$ also holds for $m\in \{10,\;21,\;32,\; 43,\;54\}$. As a result, we have, for all $n\geq 0$, 
	\begin{align*}
		A\left( 11n + 10\right) \equiv a_4(11n+10)\equiv 0 \pmod{11},
	\end{align*}
	which completes the proof of congruence (\ref{1.6}), when $j=0$.\\
	
	Similarly to prove congruence (\ref{1.7}), we consider
	\begin{align*}
		\sum_{n=0}^{\infty}a_{6}(n)q^n=&\frac{f_2^{5}}{f_1^{6}}\\
		=&\frac{f_2^{5}}{f_1^{6}}\cdot \frac{f_2^{11}}{f_2^{11}}\\
		\equiv&\frac{f_2^{16}}{f_1^{6}f_{22}}\pmod{11}.
	\end{align*}
	Define 
	\begin{align*}
		\sum_{n=0}^{\infty}B(n)q^n := \frac{f_2^{16}}{f_1^6}.
	\end{align*}
	Since $p=11$ satisfies the condition $p\equiv11\pmod{12}$, we set $r=-6$ and $s=16$ in Theorem \ref{imp} to obtain
		$$B\left( 11n + \frac{26\cdot 120}{24}\right)=
		{11}^4 B\left(\frac{n}{11}\right).$$
	This implies
	\begin{align*}
		B\left( 11n + 130\right)\equiv&\; 0\pmod{11}
	\end{align*}
	for all $n\geq 0$.  
	It is easy to verify  using Mathematica that $B(m)\equiv 0 \pmod{11}$ for $m\in \{9,\;20,\; 31,\; 42, \;53,\; 64, \;75,\;86,\;97,\;108,\;119\}$. Hence, for all $n\geq 0$, 
	\begin{align*}
		B\left( 11n + 9\right) \equiv a_6(11n+9)\equiv 0 \pmod{11},
	\end{align*}
	which completes the proof of congruence (\ref{1.7}), when $j=0$.\\
	To prove congruence (\ref{1.8}), we consider
	\begin{align*}
		\sum_{n=0}^{\infty}a_{8}(n)q^n=&\frac{f_2^{7}}{f_1^{8}}\\
		=&\frac{f_2^{7}}{f_1^{8}}.\frac{f_1^3}{f_1^3}\\
		=&\frac{f_1^3f_2^7}{f_1^{11}}\\
		\equiv&\frac{f_1^3f_2^7}{f_{11}}\pmod{11}.
	\end{align*}
	Define
	\begin{align*}
		\sum_{n=0}^{\infty}C(n)q^n := {f_1^3}{f_2^7}.
	\end{align*}
	Since $p=11$ satisfies the condition $p\equiv7, 11\pmod{12}$ and further $p\equiv3\pmod{8}$, we set $r=3$ and $s=7$ in Theorem \ref{imp} to obtain 
	$$C\left( 11n + \frac{17\cdot 120}{24}\right)=-{11}^4 C\left(\frac{n}{11}\right).$$
	This implies
	\begin{align*}
		C\left( 11n + 85\right)\equiv&\; 0\pmod{11}
	\end{align*}
	for all $n\geq 0$.  
	It is easy to verify  using Mathematica that $C(m)\equiv 0 \pmod{11}$ for $m\in \{8,\;19,\;30,\;41,\;52,\;63,\;74\}$.  
	Thus, we conclude that, for all $n\geq 0$, 
	\begin{align*}
		C\left( 11n + 8\right) \equiv a_8(11n+8)\equiv 0 \pmod{11}.
	\end{align*}
	This completes the proof of congruence (\ref{1.6}), when $j=0$.
	
	Now for all $j\geq0$ and $k \geq0$,
	\begin{align*}
		\sum_{n=0}^{\infty} a_{11j+k}(n) q^n
		= &\frac{f_2^{11j+k-1}}{f_1^{11j+k}}\\
		\equiv&\frac{f_{22}^j}{f_{11}^j} \cdot \frac{f_2^{k-1}}{f_1^k}
		\pmod{11}\\
		= &\frac{f_{22}^j}{f_{11}^j} \sum_{n=0}^{\infty} a_k(n) q^n .
	\end{align*}
	Thanks to the generating function of $a_k(n)$ and 
	the above proofs when $j=0$, the result then follows.
\end{proof}
\section{Concluding Remarks}
Much the same argument as the one given in the proof of Theorem \ref{t:2.1.9} may be used to give an alternate proof of Hirschhorn and Sellers' congruences \eqref{c2}-\eqref{c4}.
For example, to obtain (\ref{c2}), we consider 
\begin{align*}
	\sum_{n=0}^{\infty}a_{3}(n)q^n=\frac{f_2^{2}}{f_1^{3}}
	=\frac{f_2^{7}}{f_1^{14}}.\frac{f_1^{11}}{f_2^5}
	\equiv\frac{f_{14}}{f_{7}^2}.\frac{f_1^{11}}{f_2^5}\pmod{7}.
\end{align*}
With $(r,s)=(11,-5)$ in Theorem \ref{imp}, and arguing as in Section 4, we complete the proof of (\ref{c2}).

		\bigskip
		
		\bigskip
		
		\noindent\textsuperscript{1,3}Ramanujan School of Mathematical Sciences,\\Department of Mathematics,\\ Pondicherry University,\\ Puducherry - 605014, India. \\

		\bigskip
		\noindent\textsuperscript{2}Department of Mathematics and Statistics, \\University of Minnesota Duluth,\\Duluth, MN 55812, USA.\bigskip\\
		
		\noindent Email: \texttt{anjelinvallialil@pondiuni.ac.in} \\
		Email: \texttt{jsellers@d.umn.edu} (\Letter)\\
		Email: \texttt{dr.fathima.sn@pondiuni.ac.in} 
\end{document}